\documentclass[11pt]{amsart}   	% use "amsart" instead of "article" for AMSLaTeX format
\usepackage{geometry}                		% See geometry.pdf to learn the layout options. There are lots.
\usepackage{graphicx}				% Use pdf, png, jpg, or eps§ with pdflatex; use eps in DVI mode
\usepackage{mathrsfs}  
\usepackage{hyperref}
\usepackage{xcolor}
\usepackage{tikz-cd}		
\usepackage{amsfonts,amsthm}
\usepackage{mathtools}									
\usepackage{array}					
\usepackage{amssymb}
\usepackage{amsmath}
\usepackage{enumitem}						
																
\newtheorem{theorem}{Theorem}[section]
\newtheorem{proposition}[theorem]{Proposition}

\newtheorem{lemma}[theorem]{Lemma}

\newtheorem{definition}[theorem]{Definition}
\newtheorem{example}[theorem]{Example}

\title[]{On Thompson knot theory and conjugacy classes of Thompson's group $F$}

\address{Division of Mathematics \&
Research Center for Pure and Applied Mathematics,
Graduate School of Information Sciences,
Tohoku University, 6-3-09 Aramaki-Aza-Aoba, Aoba-ku, Sendai 980-8579, Japan.}
\email{yybao@tohoku.ac.jp}
\address{%
The University of Osaka, International College, 
1-2 Machikaneyama-cho, Toyonaka, Osaka 560-0043}
\email{xs54standrews@gmail.com}

\author{Yuanyuan Bao and Xiaobing Sheng}

\subjclass[]{Primary 57K10, Secondary 20F65%
}
\keywords{%
	 Thompson's group $F,$ Thompson knot theory, conjugacy class}
\usetikzlibrary{decorations.markings}

\tikzset{->-/.style={decoration={
  markings,
  mark=at position .5 with {\arrow{>}}},postaction={decorate}}}

\tikzset{-<-/.style={decoration={
  markings,
  mark=at position .5 with {\arrow{<}}},postaction={decorate}}}

\begin{document}
\maketitle
%
%
%
%abstract
\begin{abstract} 
%Thompsonfieldtheory
Jones introduced a method to produce unoriented links 
from elements of the Thompson's group $F$, 
and proved that any link can be produced by this construction. 
In this paper, we attempt to 
investigate the relations between conjugacy classes of the group $F$ 
and the links being constructed. For each unoriented link $L$, 
we find a sequence of elements of $F$ from distinct conjugacy classes 
which yield $L$ via Jones's construction. 
We also show that a sequence of $2$-bridge links can be constructed 
from elements in the conjugacy class of $x_0$ (resp. $x_1$).
\end{abstract}

\maketitle

\tableofcontents

\section{Introduction}

Thompson's groups were first constructed in 1960's by R. Thompson 
and were found to have many interesting properties 
and in connection 
with not only pure mathematics, 
but also other scientific fields
 including quantum field theory and cryptography.
They can be considered as 
subgroups of the group of homeomorphisms of the Cantor set 
whose algebraic properties and topological properties 
are being much investigated \cite{MR1426438}.

V. Jones inspired by the results on unitary representations of Thompson's groups 
\cite{key3589908m}, 
tried to build up the connection between his subfactor theory and Thompson's groups. 
%Brothier
He also found a way to construct knots and links from Thompson's group $F$ concretely, 
as an analogue of constructing knots and links from the braid groups \cite{key3986040m}. 
He then proved that all links can be constructed from Thompson's group $F$  
(an analogue of the Alexander theorem in braid group).  
Aiello proved the theorem for the oriented Thompson's group $\vec{F}$
which appears as a subgroup of $F$ \cite{Aiello:2020aa}.
Over the years, knot theoretic aspects of the construction have been studied such as in
\cite{Aiello:2022aa, Aiello:2022ad, doi:10.1142/S0218216523400230} 
and Aiello has summarised the theory as Thompson knot theory.

For braid groups, 
the Markov theorem states that two braids produce the same link if and only if 
they are connected by conjugating operations and (de)stabilizations. %%%
As a corollary, 
we see that braids belonging to the same conjugacy class produce the same link. 
The converse is obviously not true. 
One link can be produced from braids from different conjugacy classes.
% (which are connected by (de)stabilizations). 

%
One may want to know 
if we have a version of Markov theorem for Thompson's group $F$. %
Namely, for two elements in $F$ that produce the same link, 
can one figure out the operations connecting these two elements? %rewrite
In this paper, to provide more references for solving this problem, 
we investigate the relation between conjugacy classes of the group $F$ 
and the knots and links being constructed. 
We found out that for each unoriented link $L$, 
one can find a sequence of elements of Thompson's group $F$ 
from distinct conjugacy classes 
which produce $L$ via Jones's construction (Theorem \ref{thm1}). 
Unlike in the braid group case, 
we also show that for $x_0$ (resp. $x_1$) in $F$, 
one can construct infinitely many different links from the conjugacy classes of 
$x_0$ (resp. $x_1$) (Theorem \ref{thm2}). 
Note that $x_0$ and $x_1$ generate $F$. 

From the above result, 
we can further consider the following questions. 
For a given link $L$ and a given conjugacy class of $F$, 
can one find an element in the given conjugacy class that produces $L$? 
Can one construct all links from the conjugacy classes of $F$, 
rather than from the elements of $F$? 
We hope to further work on these questions in the future.
%Motivated by an attempt to tackle the Thompson's group version of the Markov's theorem, 
%we investigate the relations between conjugacy classes of the group $F$ 
%and the knots and links being constructed.
%

%We found that for each unoriented link $L,$ 
%one can find a sequence of elements of Thompson's group $F$ from 
%infinitely many distinct conjugacy classes 
%such that yield $\mathcal{L}$ via Jones' construction. 
%We also list up a few interesting examples 
%and prove a weaker complementary version of the above result.

%
\medskip
The organization of the paper is as follows.
We introduce Thompson's group $F,$ and Jones's construction of links 
in Sections \ref{sec2.1} and \ref{sec2.2}, 
and we provide Belk and Matucci's description of conjugacy classes of $F$ 
in Section \ref{sec2.3}.  
In Section \ref{sec3}, we prove Theorem \ref{thm1}. 
Finally, we prove Theorem \ref{thm2} in Section \ref{sec4}.

\section{Preliminary}
\subsection{Thompson's group $F$}
\label{sec2.1}
Thompson's group $F$ being one of the most mysterious groups 
in geometric group theory, 
can be defined from several different aspects. 
To fit with our later description of constructing knots and links 
from elements of Thompson's group $F$, 
we introduce the group $F$ 
from the viewpoint of tree diagrams.

%
%tree pairs
Let $\mathcal{T}$ be the set of rooted finite binary trees. 
We take $T_{+},$ $T_{-} \in \mathcal{T}$ 
that are the trees with the same number of leaves to form a tree pair ($T_{+}, T_{-})$ 
(see Fig. \ref{fig1} (left)),
where we call $T_{+}$ the source tree and $T_{-}$ the target tree. 

We take a tree pair $(T_{+}, T_{-})$ and flip the target tree $T_{-}$ vertically 
and we obtain $T_{-}^{flip}.$
Then we align the source tree $T_{+}$ 
and the flipped target tree $T_{-}^{flip}$ vertically as in Fig. \ref{fig1} (middle).
%Here we abuse the notation, we still use $T_{-}$ instead of $T_{-}^{flip}$ for the target tree. 
We then align all the leaves in each tree to a horizontal line 
and attach the leaves in the source tree to the corresponding leaves 
in the target tree to form 
the so-called tree diagram (Fig. \ref{fig1} (right)). 
Here we abuse the notation and still denote the tree diagram by $(T_+, T_-)$ 
\cite{Aiello:2020aa, key3589908m}. 
In a tree diagram, the source tree and target tree have common leaves.

\begin{figure}[htp]
\begin{center}
\begin{tikzpicture}[baseline=-0.65ex, thick, scale=0.3]
\draw (0, 4) to (-4, 0);
\draw (0, 4) to (4, 0);
\draw (3, 1) to (2, 0);
\draw (-3, 1) to (-2, 0);
\draw (-2, 2) to (0, 0);
\draw (0,-2) node{$T_+$};
\end{tikzpicture}
\begin{tikzpicture}[baseline=-0.65ex, thick, scale=0.3]
\draw (0, 4) to (-4, 0);
\draw (0, 4) to (4, 0);
\draw (1, 3) to (-2, 0);
\draw (0, 2) to (2, 0);
\draw (1, 1) to (0, 0);
\draw (0,-2) node{$T_-$};
\end{tikzpicture}\quad\quad
\begin{tikzpicture}[baseline=-0.65ex, thick, scale=0.3]
\draw (-1, 0) [->]to (0, 0);
\end{tikzpicture}\quad\quad
\begin{tikzpicture}[baseline=-0.65ex, thick, scale=0.3]
\draw (0, 4.5) to (-4, 0.5);
\draw (0, 4.5) to (4, 0.5);
\draw (3, 1.5) to (2, 0.5);
\draw (-3, 1.5) to (-2, 0.5);
\draw (-2, 2.5) to (0, 0.5);
\draw (0, -4.5) to (-4, -0.5);
\draw (0, -4.5) to (4, -0.5);
\draw (1, -3.5) to (-2, -0.5);
\draw (0, -2.5) to (2, -0.5);
\draw (1, -1.5) to (0, -0.5);
%\draw [dotted]  (-5, 0) to (5, 0);
\end{tikzpicture}\quad\quad
\begin{tikzpicture}[baseline=-0.65ex, thick, scale=0.3]
\draw (-1, 0) [->]to (0, 0);
\end{tikzpicture}\quad\quad
\begin{tikzpicture}[baseline=-0.65ex, thick, scale=0.3]
\draw (0, 4) to (-4, 0);
\draw (0, 4) to (4, 0);
\draw (3, 1) to (2, 0);
\draw (-3, 1) to (-2, 0);
\draw (-2, 2) to (0, 0);
\draw (0, -4) to (-4, 0);
\draw (0, -4) to (4, 0);
\draw (1, -3) to (-2, 0);
\draw (0, -2) to (2, 0);
\draw (1, -1) to (0, 0);
\draw [dotted]  (-5, 0) to (5, 0);
\end{tikzpicture}
 \caption{A tree pair aligned vertically to form a tree diagram.
 \label{fig1}}
 \end{center}
\end{figure}
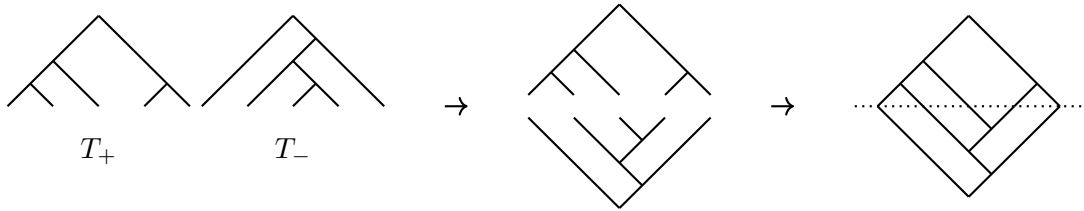

%Let $\Gamma_{\mathscr{T}}$ be the set of all such $\Gamma$-graphs.
%%%%%!!!!!!!!!!!!!!!!!!!!!!!!!!!!!

%
%

%{\color{red} There is a figure there}
%
%\begin{figure}[htp]
%\centerline{
%\includegraphics[width=2.5in]{fig6.jpeg}}
% \caption{$\Gamma$ graph}
% \label{fig2}
%\end{figure}

%
%
%
%For each tree pair, 
%the leaves in the source tree are aligned in an order 
%corresponding to the ones of target tree.
%Take a leaf $\gamma_i$  in the source tree $T_{+}$ of a tree pair $(T_{+}, T_{-})$ 
%representing an element $g \in F,$
%and attach a rooted binary tree with only two edges (we call it a caret), 
%and we find the corresponding leaf $\gamma_{g(i)}$ in the target tree $T_{-}$ 
%and attach the same binary tree to $\gamma_{g(i)}$ (see Figure \ref{fig1}).
%The $(T'_{+}, T'_{-})$ we obtained represent the same element $g$ 
%but are formed by trees with one more leaf in each tree. 
For each tree diagram we introduce the following operation.
When we replace a leaf on the horizontal line with a pair of carets (see Fig. \ref{fig2}), 
we have constructed a new tree diagram. 
%which represents the same group element. 

%
\begin{figure}[htp]
\begin{center}
\begin{tikzpicture}[baseline=-0.65ex, thick, scale=0.4]
\draw (0, 4) to (-4, 0);
\draw (0, 4) to (4, 0);
\draw (3, 1) to (2, 0);
\draw (-3, 1) to (-2, 0);
\draw (-2, 2) to (0, 0);
\draw (0, -4) to (-4, 0);
\draw (0, -4) to (4, 0);
\draw (1, -3) to (-2, 0);
\draw (0, -2) to (2, 0);
\draw (1, -1) to (0, 0);
\draw [dotted]  (-5, 0) to (5, 0);
\end{tikzpicture}
\begin{tikzpicture}[baseline=-0.65ex, thick, scale=0.3]
\draw (-1, 0) [->]to (0, 0);
\end{tikzpicture}\quad\quad
\begin{tikzpicture}[baseline=-0.65ex, thick, scale=0.4]
\draw (0, 4.5) to (-4.5, 0);
\draw (0, 4.5) to (4.5, 0);
\draw (3, 1.5) to (-0.5, -2) ;
\draw (-3, 1.5) to (-2, 0.5);
\draw (-2, 2.5) to (1, -0.5);
\draw (0, -4.5) to (-4.5, 0);
\draw (0, -4.5) to (4.5, 0);
\draw (1, -3.5) to (-2, -0.5);
%\draw (0, -2.5) to (2, -0.5);
%\draw (1, -1.5) to (0, -0.5);
\draw [red] (-2, 0.5) to (-2.5, 0);
\draw [red] (-2, -0.5) to (-2.5, 0);
\draw [red] (-2, -0.5) to (-1.5, 0);
\draw [red] (-2, 0.5) to (-1.5, 0);
\draw [dotted]  (-5, 0) to (5, 0);
\end{tikzpicture}
 \caption{An elementary expansion.}
 \label{fig2}
 \end{center}
 \end{figure}
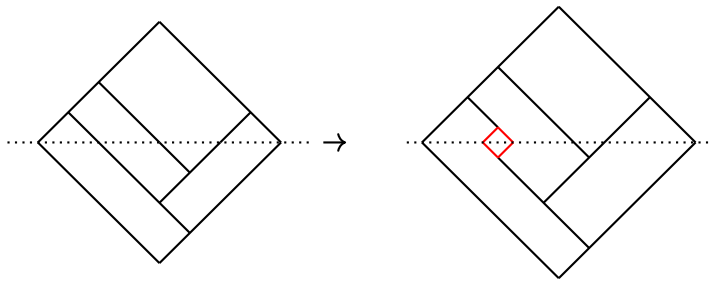

We call this process an {\it elementary expansion}. 
The reversing process is called {\it elementary reduction}. 
Here, a {\it caret} is a binary tree with one root and two leaves 
which is a building block of the binary tree. 
 \begin{center}
\begin{tikzpicture}[baseline=-0.65ex, thick, scale=0.4]
\draw (-3, 1) to (-4, 0);
\draw (-3, 1) to (-2, 0);
\end{tikzpicture}
 \end{center}
 %
%This process is called elementary reduction and 
%Together 
%with elementary expansion, 
%we can define an equivalence relation $\sim$ on $\%{\Gamma_{\mathscr{T}}\}$ 
%where $\Gamma_{(T_{+}, T_{-})} \sim \Gamma_{(T'_{+}, T'_{-})}.$ 
A tree diagram is said to be {\it reduced} if no further elementary reductions can be applied to it. 
Two tree diagrams are said to be {\it equivalent} if they
are related to each other by applying finitely many elementary expansions and reductions. 
%The group elements of $F$ are defined by the equivalent classes 
%$\Gamma_{\mathscr{T}} / \sim.$ 
%
Taking a tree diagram, if we keep performing elementary reductions, 
we will obtain a unique reduced tree diagram
which gives a representative of the equivalent class \cite{MR1426438, key3589908m}.%Unique

We then describe the multiplication of two tree diagrams %tree pairs 
which provide the group operation. %
Take two tree diagrams $(T_+, T_-)$ and $(T'_+, T'_-)$
and place them vertically 
such that the root of $T_-$ and the root of $T'_+$ can be connected by a vertical line.
The multiplication process requires us to replace the $X$ shaped graph 
with two vertical lines, as shown in Fig. \ref{replace}.

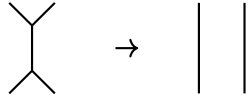
\begin{figure}[htp]
\begin{center}
\begin{tikzpicture}[baseline=-0.65ex, thick, scale=0.3]
\draw (-1, 2) to (0, 1);
\draw (1, 2) to (0, 1);
\draw (0, -1) to (0, 1);
\draw (-1, -2) to (0, -1);
\draw (1, -2) to (0, -1);
\end{tikzpicture}\quad\quad
\begin{tikzpicture}[baseline=-0.65ex, thick, scale=0.3]
\draw (-1, 0) [->]to (0, 0);
\end{tikzpicture}\quad\quad
\begin{tikzpicture}[baseline=-0.65ex, thick, scale=0.3]
\draw (-1, 2) to (-1, -2);
\draw (1, 2) to (1, -2);
\end{tikzpicture}
\end{center} 
\caption{Replacing the local diagram on the left by two vertical lines.}
 \label{replace}
\end{figure}

Then we adjust the graph 
and a reduced tree diagram can then be obtained 
by performing elementary reductions. 
See Fig. \ref{fig3} for an example.
%
%The group operations 
%

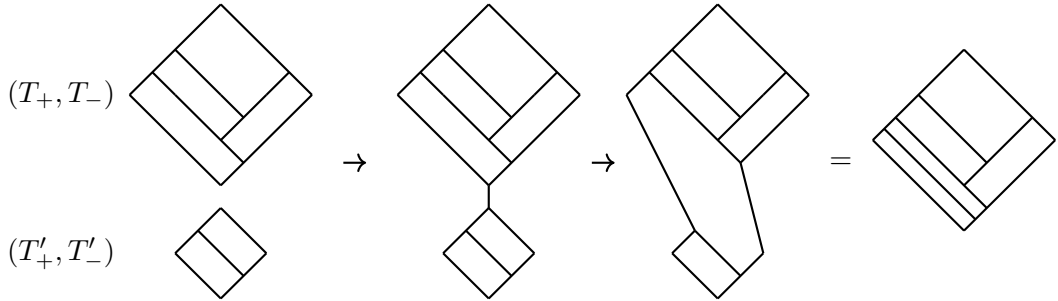
\begin{figure}[htp]
\begin{tikzpicture}[baseline=-0.65ex, thick, scale=0.3]
\draw (0, 4) to (-4, 0);
\draw (0, 4) to (4, 0);
\draw (3, 1) to (2, 0);
\draw (-3, 1) to (-2, 0);
\draw (-2, 2) to (0, 0);
\draw (0, -4) to (-4, 0);
\draw (0, -4) to (4, 0);
\draw (1, -3) to (-2, 0);
\draw (0, -2) to (2, 0);
\draw (1, -1) to (0, 0);
%\draw [dotted]  (-5, 0) to (5, 0);
\draw (-7, 0) node{$(T_+, T_-)$};
\draw (0, -5) to (-2,-7);
\draw (0, -5) to (2, -7);
\draw (0, -9) to (-2,-7);
\draw (0, -9) to (2, -7);
\draw (-1, -6) to (1, -8);
\draw (-7, -7) node{$(T'_+, T'_-)$};
\end{tikzpicture}\quad
\begin{tikzpicture}[baseline=-0.65ex, thick, scale=0.3]
\draw (-1, -3) [->]to (0, -3);
\end{tikzpicture}\quad
\begin{tikzpicture}[baseline=-0.65ex, thick, scale=0.3]
\draw (0, 4) to (-4, 0);
\draw (0, 4) to (4, 0);
\draw (3, 1) to (2, 0);
\draw (-3, 1) to (-2, 0);
\draw (-2, 2) to (0, 0);
\draw (0, -4) to (-4, 0);
\draw (0, -4) to (4, 0);
\draw (1, -3) to (-2, 0);
\draw (0, -2) to (2, 0);
\draw (1, -1) to (0, 0);
%\draw [dotted]  (-5, 0) to (5, 0);
%\draw (-7, 0) node{$(T_+, T_-)$};
\draw (0, -5) to (-2,-7);
\draw (0, -5) to (2, -7);
\draw (0, -9) to (-2,-7);
\draw (0, -9) to (2, -7);
\draw (-1, -6) to (1, -8);
\draw (0, -4) to (0, -5);
%\draw (-7, -7) node{$(T'_+, T'_-)$};
\end{tikzpicture}
\begin{tikzpicture}[baseline=-0.65ex, thick, scale=0.3]
\draw (-1, -3) [->]to (0, -3);
\end{tikzpicture}
\begin{tikzpicture}[baseline=-0.65ex, thick, scale=0.3]
\draw (0, 4) to (-4, 0);
\draw (0, 4) to (4, 0);
\draw (3, 1) to (2, 0);
\draw (-3, 1) to (-2, 0);
\draw (-2, 2) to (0, 0);
%\draw (0, -4) to (-4, 0);
\draw (1, -3) to (4, 0);
\draw (1, -3) to (-2, 0);
\draw (0, -2) to (2, 0);
\draw (1, -1) to (0, 0);
%\draw [dotted]  (-5, 0) to (5, 0);
%\draw (-7, 0) node{$(T_+, T_-)$};
\draw  (-1, -6) to (-2,-7);
%\draw (0, -5) to (2, -7);
\draw (0, -9) to (-2,-7);
\draw (0, -9) to (2, -7);
\draw (-1, -6) to (1, -8);
\draw (-1, -6) to (-4, 0);
\draw (2, -7) to (1, -3);
%\draw (-7, -7) node{$(T'_+, T'_-)$};
\end{tikzpicture}
\begin{tikzpicture}[baseline=-0.65ex, thick, scale=0.3]
\draw (-1, -3) node{$=$};
\end{tikzpicture}
\begin{tikzpicture}[baseline=-0.65ex, thick, scale=0.3]
\draw (0, 2) to (-4, -2);
\draw (0, 2) to (4, -2);
\draw (0, -6) to (-4, -2);
\draw (0, -6) to (4,-2);
\draw (0.5, -5.5) to (-3.5, -1.5);
\draw (-3, -1) to (1, -5);
\draw (-2, 0) to (1, -3);
\draw (0, -4) to (3, -1);
\end{tikzpicture}
 \caption{The multiplication of two tree diagrams.}
 \label{fig3}
 \end{figure}

\begin{theorem}[Thompson's group $F$]
The set of equivalent classes of tree diagrams together with 
the multiplication described is isomorphic to the Thompson's group $F.$
\end{theorem}

%
%\begin{remark}
%A $\Gamma$ graph can be regarded as a ``vertical" version of a tree pair representation 
%of a group element and the description of uniquely reduced graph coincides 
%with the description of uniquely reduced tree pair representing elements of $F.$  
%They essentially carry the same information as a tree pair, but visually, 
%they fit with the construction of knots and links better
%\cite[Proposition 4.1.1]{key3589908m}.
%\end{remark}

%
From now on, we regard a tree diagram as an element of $F$. 
The group $F$ has finiteness properties 
\cite{MR885095} 
and a really nicely computable infinitely generating set $\{x_i\}_{i=0}^{\infty}$.
See Fig. \ref{ffig3} for the illustration of $x_i$'s,
where the first two $x_0$ and $x_1$ in the sequence form a finite generating set. 
For every reduced tree diagram, there is a unique word 
with respect to this infinite generating set,
and the word can be read off directly from the tree diagram \cite{MR1426438}. 
We omit the details here.
%%%figure 
%
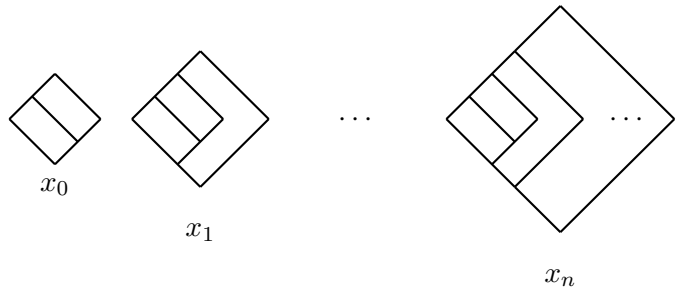
\begin{figure}[htp]
\begin{tikzpicture}[baseline=-0.65ex, thick, scale=0.3]
\draw (0, 2) to (-2,0);
\draw (0, 2) to (2, 0);
\draw (0, -2) to (-2,0);
\draw (0, -2) to (2, 0);
\draw (-1, 1) to (1, -1);
\draw (0, -3) node{$x_0$};
\end{tikzpicture}\quad
\begin{tikzpicture}[baseline=-0.65ex, thick, scale=0.3]
\draw (1, 3) to (-2,0);
\draw (0, 2) to (2, 0);
\draw (1, -3) to (-2,0);
\draw (0, -2) to (2, 0);
\draw (-1, 1) to (1, -1);
\draw (1, 3) to (4,0);
\draw (1, -3) to (4,0);
\draw (1, -5) node{$x_1$};
\end{tikzpicture}\quad\quad
\begin{tikzpicture}[baseline=-0.65ex, thick, scale=0.3]
\draw (0, 0) node{$\cdots$};
\end{tikzpicture}\quad\quad
\begin{tikzpicture}[baseline=-0.65ex, thick, scale=0.3]
\draw (1, 3) to (-2,0);
\draw (0, 2) to (2, 0);
\draw (1, -3) to (-2,0);
\draw (0, -2) to (2, 0);
\draw (-1, 1) to (1, -1);
\draw (1, 3) to (4,0);
\draw (1, -3) to (4,0);
\draw (3, 5) to (1,3);
\draw (3,5) to (8,0);
\draw (3, -5) to (1,-3);
\draw (3,-5) to (8,0);
\draw (6, 0) node{$\cdots$};
\draw (3, -7) node{$x_n$};
\end{tikzpicture}
\caption{The infinite generating set of $F$. 
The number of leaves in $x_i$ is $i+3$.}
\label{ffig3}
\end{figure}

\subsection{Jones's construction}
\label{sec2.2}
In \cite{key3589908m, key3986040m} Jones introduced two equivalent methods 
to produce unoriented knots and links from elements of the Thompson's group $F$. 
Here we briefly review both of them. 
For more details, 
please see Jones's original papers or the nice survey \cite{Aiello:2022ac}. 

%
% from group element to Thompson graph
\subsubsection{Constructing Thompson Tait graphs}

%
% from Thompson graph to Tait diagram
Taking a tree diagram $(T_{+}, T_{-}),$ 
we construct the corresponding Tait graph $\Gamma (T_{+}, T_{-})$ as follows.

We have identified the leaves of $T_{+}$ and $T_{-}$ 
by putting them on a horizontal line. 
Next we label their leaves by $\{\gamma_0, \gamma_1, \cdots, \gamma_{n-1}\}$ 
from left to right, 
where $n$ is the number of leaves in $T_{+}$ (and thus in $T_{-}$).
For $\Gamma (T_{+}, T_{-})$, all vertices are drawn on the same horizontal line. 
Then we draw one vertex in front of $\gamma_0$ and denote it by $v_0,$ 
and for the rest, we draw one vertex each between two aligned vertices in 
$\{\gamma_0, \gamma_1, \cdots, \gamma_{n-1}\}$ 
and denote them by $\{v_1, v_2, \cdots, v_{n-1}\}$ from left to right. 

For edges of $\Gamma (T_{+}, T_{-})$, 
we first focus on the tree $T_{+}$. 
For each edge in $T_{+}$ with positive slope, 
we draw an edge $e$ passing across it 
so that $e$ connecting two vertices in $\{v_0, v_1, \cdots, v_{n-1}\}$ 
and has no intersection with the other part of $(T_{+}, T_{-})$.  
For each edge in $T_{-}$ with negative slope, 
we draw an edge similarly. 
See Fig. \ref{fig7} for an example of constructing Tait graph. 
Moreover, we assign a positive sign to each edge constructed from $T_{+}$ 
and a negative sign to each edge constructed from $T_{-}$. 

\begin{figure}[htp]
\centerline{
\includegraphics[width=6in]{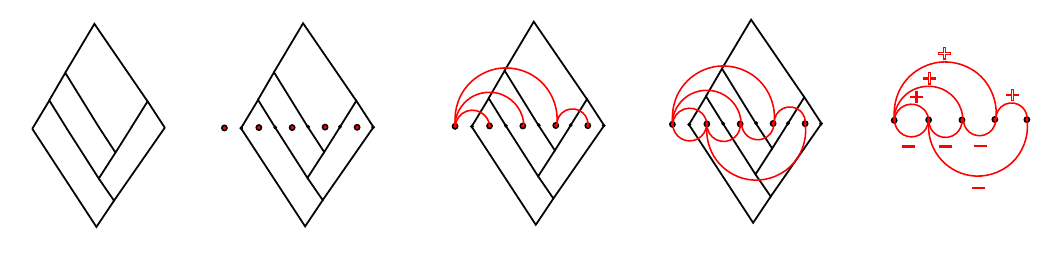}}
 \caption{From a Tait graph to a link diagram. \label{fig7}}
\end{figure}

Now since $\Gamma (T_{+}, T_{-})$ is a signed planar graph, 
we follow the standard method in knot theory to construct a link diagram. 
For an edge with a positive or negative sign, 
we assign a crossing as in Fig. \ref{tait}. See Fig. \ref{fig8} for an example. 
More details can be found in \cite{key3986040m}.
\begin{figure}[htp]
\begin{center}
\begin{tikzpicture}[baseline=-0.65ex, thick, scale=1]
\draw [red] (-1, 0) to (1,0);
\draw (-1, 0) node{$\bullet$};
\draw  (1, 0) node{$\bullet$};
\draw [red] (-1.5, 0)  node{$+$};
\draw (-0.5, -0.5)  to (0.5, 0.5);
\draw (0.5, -0.5)  to (0.2, -0.2);
\draw (-0.5, 0.5)  to (-0.2, 0.2);
\end{tikzpicture}\quad \quad\quad
\begin{tikzpicture}[baseline=-0.65ex, thick, scale=1]
\draw [red] (-1, 0) to (1,0);
\draw (-1, 0) node{$\bullet$};
\draw (1, 0) node{$\bullet$};
\draw [red] (-1.5, 0)  node{$-$};
\draw (-0.5, 0.5)  to (0.5, -0.5);
\draw (-0.5, -0.5)  to (-0.2, -0.2);
\draw (0.5, 0.5)  to (0.2, 0.2);
\end{tikzpicture}
 \caption{Assign a crossing to an edge of the Tait graph. \label{tait}}
 \end{center}
\end{figure}
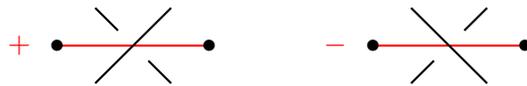

\begin{figure}[htp]
\centerline{
\includegraphics[width=5in]{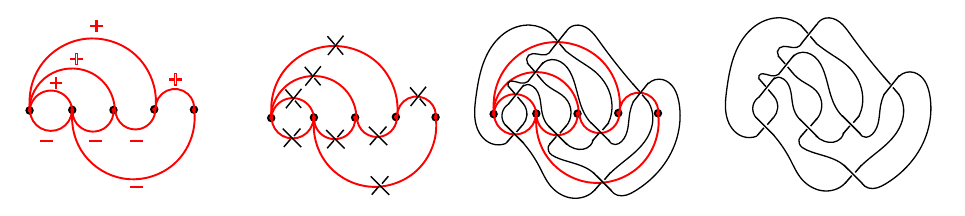}}
 \caption{From a Tait graph to a link diagram. \label{fig8}}
\end{figure}

\subsubsection{An equivalent way of the construction}
There is an equivalent way to construct knots and links from elements of Thompson's group $F$ 
without going through the Tait graphs \cite{key3986040m}. %%Jones 

We take again an element $g$ of Thompson's group $F$ 
represented by a tree diagram $(T_{+}, T_{-})$.
The leaves of $T_{+}$ and $T_{-}$ sit on a horizontal line. 
We label them by $\{\gamma_0, \gamma_1, \cdots, \gamma_{n-1}\}$ from left to right. 
Now we insert a new vertex $w_i$ between $\gamma_{i-1}$ and $\gamma_{i}$, 
for $i=0, 1, \cdots, n-1$. 
For each $w_i$, we add an edge passing through $w_i$
from a trivalent vertex or the root of $T_{+}$ to a trivalent vertex or the root of $T_{-}$ 
in the only possible planar way. 
The last step is to take a closure of diagram by 
connecting the root of $T_{+}$ and the root of $T_{-}$. 
Now all the trivalent vertices are turned into $4$-valent. 
See Fig. \ref{fig01} for an example. 

\begin{figure}[htp]
\centerline{
\includegraphics[width=5in]{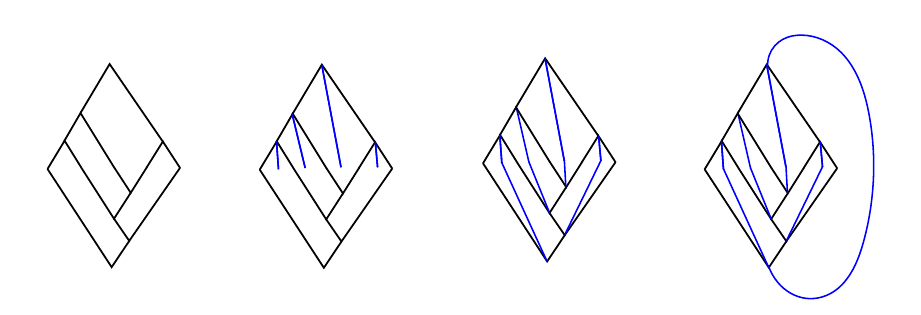}}
 \caption{Connecting vertices of $T_{+}$ and $T_{-}$.}
 \label{fig01}
\end{figure}

Next we can get a link diagram by replacing each $4$-valent %or trivalent vertex 
by a crossing 
so that the two edges connecting a node to its two descendants in $T_{+}$ and $T_{-}$ 
become the overpass. 
We then obtain a link diagram. 
See Fig. \ref{fig02} for an example. 

\begin{figure}[htp]
\centerline{
\includegraphics[width=2.5in]{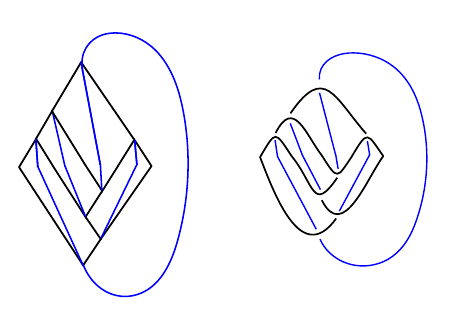}}
 \caption{Changing each $4$-valent vertex into a crossing.}
 \label{fig02}
\end{figure}

The procedure above gives us a concrete way %concrete?
to construct a link from a tree diagram. 
It is easy to see that if two tree diagrams are connected by an elementary reduction, 
the two links obtained from them only differ by a trivial component. 
Therefore we have the following map 
\begin{eqnarray*}
J: F &\to& \{\text{unoriented links}\}\\
g&\mapsto& \mathcal{L}_g
\end{eqnarray*}
Here we identify two links if they only differ by trivial components. 
More importantly, Jones proved that the map above is surjective. 

\begin{theorem}[Jones \cite{key3986040m}]
Given an unoriented link $L$, 
there exists an element $g\in F$ such that $\mathcal{L}_g=L$.
\label{jones}
\end{theorem}
The map $J$ is far from being injective, 
as we could see in the following example. 
A version of Markov theorem for Thompson's group $F$ aims to 
figure out the relations between the pre-images of a given link under $J$.
\begin{example}
For a tree diagram $g=(T_+, T_-)$, 
we choose a leaf and add the tree diagram representing $x_0$ to this leaf 
to get a new tree diagram $(T'_+, T'_-)$, 
as in Fig. \ref{ffig9} (left). 
Then $(T_+, T_-)$ and $(T'_+, T'_-)$ correspond to different elements in Thompson's group, 
but they are sent to the same unoriented link by $J$. Indeed, adding $x_0$ to a leaf amounts to 
taking a connected sum of $\mathcal{L}_g$ with a trivial knot,
as we can see in Fig. \ref{ffig9} (right). 
\end{example}

\begin{figure}[htp]
\centerline{\includegraphics[width=5in]{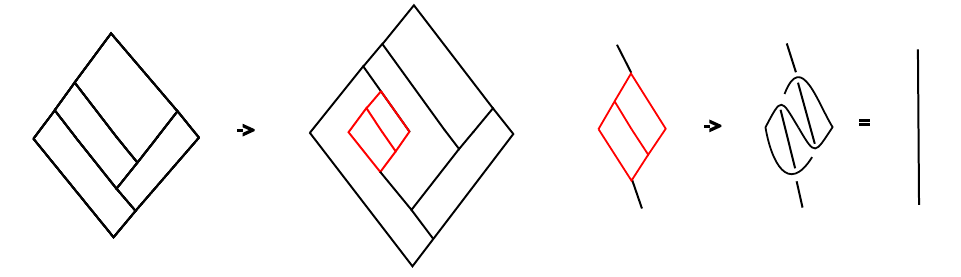}}
\caption{Different elements that produce the same link.}
\label{ffig9}
\end{figure}

%For two elements $g, h\in F$, one may want to know the relation 
%between the diagrams $\mathcal{L}_{gh}$ and $\mathcal{L}_{g}$ and $\mathcal{L}_{h}$. 
%Indeed, $\mathcal{L}_{gh}$ is the result of doing band surgeries 
%on the connected sum $\mathcal{L}_{g}\sharp \mathcal{L}_{h}$, 
%as we can see in the example in Figure \ref{fig10}. 
%Therefore, in general it is hard to read off $\mathcal{L}_{gh}$ just 
%from $\mathcal{L}_{g}$ and $\mathcal{L}_{h}$. 

%
%\begin{figure}[htp]
%\centerline{\includegraphics[width=5in]{fig10.pdf}}
%\caption{An example of $\mathcal{L}_{g}$, $\mathcal{L}_{h}$ and $\mathcal{L}_{gh}$.}
%\label{fig10}
%\end{figure}

%
%
\subsection{Conjugacy classes of $F$ and annular strand diagrams} 
\label{sec2.3}

In this section, 
we stray from the path of Jones's construction on knots and links a little 
and focus on the conjugacy problem of $F$ 
and its graphical solution. 
The conjugacy problem of Thompson's group $F$ has been solved 
\cite{MR1396957} long ago 
and solutions from this very different perspective are provided recently 
by Belk and Matucci 
\cite{Belk2014} 
using strand diagrams. 
To explain this, %theory, 
we first define the tool used here, 
the strand diagram.

%
%The ``moves" used in the strand diagrams imitate the ``moves" for Thompson knots 
%and may provide an alternative approach to our problem. 
%To explain this similarity here, we first defined the tool used here, the strand diagram:

%strand diagram
\begin{definition}[Strand diagram \cite{Belk2014}]
\rm 
A {\it strand diagram} is a directed acyclic planar graphs 
in the unit square 
satisfying the following conditions.
\begin{enumerate}
\item There is a unique univalent vertex 
which is a source on the top boundary of the unit square, 
and a unique univalent vertex 
which is a sink on the bottom boundary of the unit square;
\begin{figure}[htp]
\centerline{\includegraphics[width=1.25in]{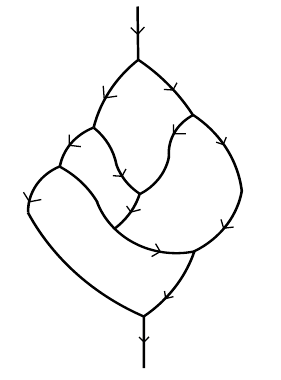}}
\caption{An example of a strand diagram}
\label{fig4}
\end{figure}
\item All other vertices are trivalent, 
and each of them is either a split or a merge.
%split_and_merge
\begin{figure}[htp]
\centerline{\includegraphics[width=1.5in]{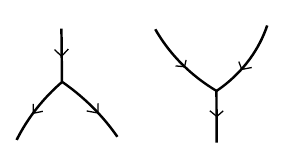}}
\caption{A split and a merge.}
\label{fig5}
\end{figure}
\end{enumerate}
\label{defsd}
\end{definition}
Two strand diagrams are said to be {\it equivalent} 
if they are connected by a finite sequence of moves of 
type I and type II in Fig. \ref{fig12}. 

A tree diagram $(T_+, T_-)$ can be identified with a strand diagram 
since we can add an up-to-down direction to each edge of $T_{+}$ and $T_{-}$. 
Then all the trivalent vertices of $T_+$ (resp. $T_-$) are splits (resp. merges). 
Under this identification, we have the following. %result.
\begin{theorem}[\cite{Belk2014}] %Belk Matucci
Thompson’s group $F$ is isomorphic to 
the group of all equivalence classes of strand diagrams, 
with product induced by concatenation.
\label{thmb2014}
\end{theorem}

To investigate the conjugacy problem in $F$, 
one needs to take the ``closure" of the strand diagram 
and to consider the annular strand diagram.

%annular_strand_diagram
\begin{definition}[Annular strand diagram \cite{Belk2014}]
\rm 
An {\it annular strand diagram} is a directed planar graph 
embedded in the annulus 
satisfying the following conditions.
\begin{enumerate}
\item Every vertex is trivalent, and is either a split or a merge;
\item Every directed cycle has positive winding number around the central hole.
\end{enumerate}
\label{defasd}
\end{definition}
\begin{figure}[htp]
\centerline{\includegraphics[width=2in]{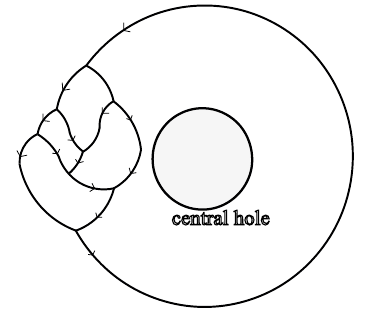}}
\caption{An example of an annular strand diagram.}
\label{fig6}
\end{figure}

Note that free loops are allowed in an annular strand diagram, 
and each free loop must wind counterclockwise 
around the central hole of the annulus.

A {\it reduction} of an annular strand diagram is either of the three moves 
in Fig. \ref{fig12}. 
An annular strand diagram is said to be {\it reduced} 
if no further reductions can be conducted. 
Two annular strand diagrams are said to be {\it equivalent} 
if they are connected by the three types of moves in Fig. \ref{fig12}. 
It is shown in \cite{Belk2014} that every annular strand diagram 
is equivalent to a unique reduced annular strand diagram. 

If we identify the top boundary and the bottom boundary of a unit square, 
we get an annulus. 
Under this identification, 
a strand diagram embedded in the unit square 
becomes an annular strand diagram in the annulus. 
Hence, a tree diagram representing an element of the Thompson's group $F$ 
corresponds to an annular strand diagram. 
We use $\mathcal{A}_h$ to denote the annular strand diagram 
corresponding to a tree diagram, %representing $h\in F$, 
and $r\mathcal{A}_h$ to denote the reduced annular strand diagram 
equivalent to $\mathcal{A}_h$. 
The following theorem provides a graphical solution to the 
conjugacy problem of $F$. 
%when two elements in $F$ are conjugate. 

%
\begin{theorem}[\cite{Belk2014}]%Belk Matucci 
Two elements $h$ and $g$ of $F$ are in the same conjugacy class 
if and only if 
$\mathcal{A}_h$ and $\mathcal{A}_g$ are equivalent, i.e., $r\mathcal{A}_h=r\mathcal{A}_g$.
\label{conj}
\end{theorem}

%reduction
%
\begin{figure}[htp]
\centerline{\includegraphics[width=5in]{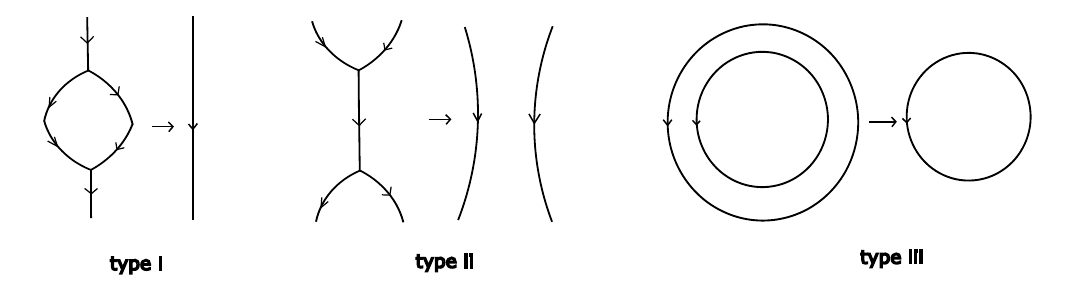}}
\caption{Three types of reductions. 
The bigon in type I move must bound a disk in the annulus. 
The two loops in type III move must be parallel free loops.}
\label{fig12}
\end{figure}

%%%%%%%%%%%%%%%%%%%%%%%%%%%%%%%%%%%%%%%%%%%%%%
%
\section{A sequence of elements from distinct conjugacy classes that produce the same link.}
\label{sec3}
%
%
%Jones' construction of knots and links are very explicit, 
%but is not compatible with the group structure 
%unlike the case for braid groups 
%which makes the exploration of Markov theorem very demanding to even start with.
%
%
%There are many examples being previously constructed 
%in order to understand the relation between the groups and 
%knots and links being construction including Jones examples mentioned in \ref{,}.
%
In this section, 
we construct a sequence of group elements and we prove Theorem \ref{thm1}.
%
%
%$\subsection{The construction of a sequence}
%\label{subsec31}
%
%Let $h \in F$ be represented by 
Given a tree diagram $(T_{+}, T_{-})$,
then we can construct a tree diagram $(H_{+}, H_{-})$ from 
$(T_{+}, T_{-})$ as follows.
%
%Let $(T_{+}, T_{-}) \in \mathscr{T}$ be the (reduced ) tree pair representing $h$ and  
Let $(T_{a(+)}, T_{a(-)})$ be the reduced tree diagram representing 
$a = x_0^{3}x_2^{-1}x_0^{-3}.$
By attaching the root of the source tree $T_{+}$ 
to the first leaf of the source tree $T_{a(+)}$ 
from the left 
%of a rooted binary tree 
%which is 
%of the reduced tree pair of $x_0,$ 
and attaching the root of $T_{-}$ %of the reduced tree pair of $x_0,$ 
to the first leaf of the target tree $T_{a(-)},$ 
we obtain a new tree diagram $(H_{+}, H_{-})$ 
(see Fig. \ref{fig15}).

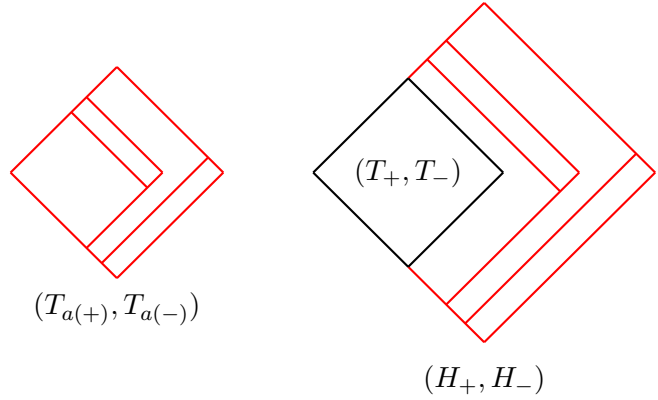
\begin{figure}[htp]
\begin{tikzpicture}[baseline=-0.65ex, thick, scale=0.2]
\draw [red] (0, 7) to (-7, 0);
\draw [red] (0, 7) to (7, 0);
\draw [red] (-1,-6) to (6, 1);
\draw [red] (0, -7) to (-7, 0);
\draw [red] (0, -7) to (7, 0);
\draw [red] (-2, -5) to (3,0);
\draw [red] (-2, 5) to (3,0);
\draw [red] (-3, 4) to (2, -1);
%\draw [dotted]  (-5, 0) to (5, 0);
\draw (0, -9) node{$(T_{a(+)}, T_{a(-)})$};
\end{tikzpicture}\quad \quad \quad
\begin{tikzpicture}[baseline=-0.65ex, thick, scale=0.25]
\draw [red] (0, 9) to (-4, 5);
\draw [red] (0, 9) to (9, 0);
\draw [red] (-1,-8) to (8, 1);
\draw [red] (0, -9) to (-4, -5);
\draw [red] (0, -9) to (9, 0);
\draw [red] (-2, -7) to (5,0);
\draw [red] (-2, 7) to (5,0);
\draw [red] (-3, 6) to (4, -1);
\draw (-9, 0) to (-4, 5);
\draw (-9, 0) to (-4, -5);
\draw (1, 0) to (-4, 5);
\draw (1, 0) to (-4, -5);
\draw (-4, 0) node{$(T_{+}, T_{-})$};
%\draw [dotted]  (-5, 0) to (5, 0);
\draw (0, -11) node{$(H_{+}, H_{-})$};
\end{tikzpicture}
\caption{Constructing $(H_{+}, H_{-})$ from $(T_{+}, T_{-})$.}
\label{fig15}
\end{figure}
%\begin{figure}[htp]
%\centering
%\subfloat[]{
%\includegraphics[width=4in]{fig142.jpeg}
% \label{fig14}
%  \label{fig14}
% }
%\caption{Constructing $h_1$ from $h_0$}
%\end{figure}
%
%
\begin{lemma}
For any tree diagram $(T_+, T_-)$, 
we have that $(T_{+}, T_{-}) \neq (H_{+}, H_{-})$ as elements in $F$.
\label{lem31}
\end{lemma}
\begin{proof}
Note that the elementary reductions occur between two adjoining leaves in a tree diagram. %between
From the construction, 
we can regard $(T_{+}, T_{-})$ as a subgraph of $(H_{+}, H_{-})$. 
By analysis on the tree diagram, %?
%we see that if an elementary reduction can be applied in $(H_{+}, H_{-})$, 
%it can only be applied in $(T_{+}, T_{-})$. 
we see that $(H_{+}, H_{-})$ is reduced if and only if $(T_{+}, T_{-})$ is reduced. 
Therefore, the reduced tree diagram equivalent to $(H_{+}, H_{-})$ 
has four more leaves then that of $(T_{+}, T_{-})$, 
from which the lemma follows. 
\end{proof}

%
%\begin{figure}[htp]
%\centering
  %\label{fig16}
% }
%\caption{}
%\end{figure}
%representing a new group element $h_{1}.$
%
%The pair $(D_{+}, D_{-})$ is reduced and $(D_{+}, D_{-}) = h_1 \neq  h_0.$ 
%This can be deduced from the reduced tree pair of the group elements \cite{}. %burillo. 
%
%
We start from a reduced tree diagram $h_{1}=(T_+, T_-)$ 
and construct a reduced tree diagram $h_{2}=(H_+, H_-)$ as above. 
Then we can repeat this process and 
construct new reduced tree diagram $h_{i+1}$ from $h_{i}$ inductively where $i \in \mathbb{N}.$
 
%$i=1, 2, 3, \cdots$. %that represent an element 
%that is different from the previous ones 
%attaching the source tree of the newly produced reduced tree pair to the first leaf of $T_{x_0(+)}$ %%
%and the target tree of the newly produced reduced tree pair to the first leaf of $T_{x_0(-)} $
%and we can produce a sequence of distinct elements $\{h_i \}_{i =0}^{\infty}$ in $F.$
%
%
\begin{lemma}
The elements $\{h_i \}_{i = 1}^{\infty}$ are all distinct in $F.$
\label{lem32}
\end{lemma}

\begin{proof}
%
%We assume that for some $n \in \mathbb{N},$ 
%$h_n$ is distinct from all $h_i$ such that $i \leq n$ 
%and $(H_{n(+)}, H_{n(-)})$ is the reduced tree pair representing of $h_n.$ 
% 
%Then we construct $h_{n+1}$ similarly, 
%where we attach the root of the source tree $H_{n(+)}$ 
%to the first leaf of the source tree $T_{a(+)}$ 
%from the left 
%and attaching the root of $H_{n(-)}$  
%to the first leaf of the target tree $T_{a(-)},$ 
%then we obtain a new tree pair $(H_{(n+1)(+)}, H_{(n+1)(-})$ 
%where we denote by $h_{n+1}$ the group element represented by $(H_{(n+1)(+)}, H_{(n+1)(-})$ 
%and the pair is reduced.
%
%Since the number of leaves in $(H_{(n+1)(+)}, H_{(n+1)(-})$ is different from 
%all other reduced tree pairs representing elements in $\{h_i \}_{i \in \mathbb{N}}$ for $i \leq n,$ 
%$h_{n+1}$ is distinct from all these elements.
%
%This follows from the fact that each element of $F$ 
%can be represented by a uniquely reduced tree pair \cite{MR1426438} 
%a reduced $\Gamma$ graph \ref{subsec31}) 
%and all reduced tree pair representing $h_i$s are distinct. 
From the proof of Lemma \ref{lem31}, 
we see that the number of leaves of the reduced tree diagram representing $h_i$ increases 
as $i$ increases, 
and the lemma follows.
%
%Therefore, 
%for any $i < j $ where $i, j \in \mathbb{N},$ 
%$h_i$ and $h_j$ has different reduced tree pair representation 
%and hence $h_i \neq h_j$ for all $i$ and $j.$
\end{proof}

\begin{proposition}
For the sequence $\{h_i \}_{i = 1}^{\infty},$ constructed above,
we have that $\mathcal{L}_{h_{i}} = \mathcal{L}_{h_1} $ holds for all $i \in \mathbb{N}.$ 
\label{prop33}
\end{proposition}

\begin{proof}
Recall the construction described in Section \ref{sec2.2}. For any $k \in \mathbb{N},$ 
we construct 
$\mathcal{L}_{h_{k+1}}$ from $\mathcal{L}_{h_{k}}$ as in Fig. \ref{fig16} 
and we see that after a few Reidemeister moves, 
$\mathcal{L}_{h_{k+1}} = \mathcal{L}_{h_{k}}.$ 
Inductively, we have $\mathcal{L}_{h_{i}} = \mathcal{L}_{h_{1}}$ for all $i \in \mathbb{N}.$ 
\end{proof}

\begin{figure}[htp]
\centerline{\includegraphics[width=4in]{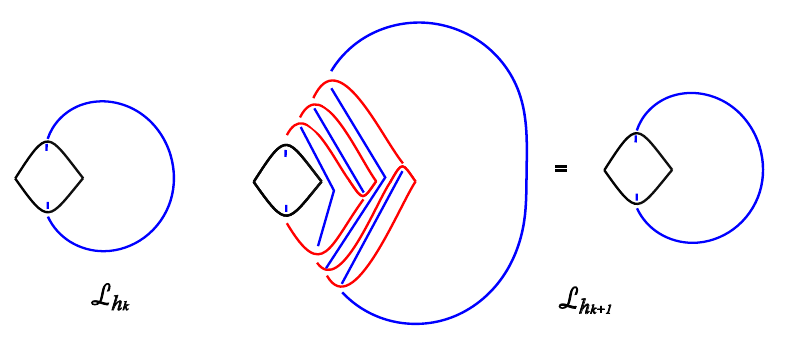}}
\caption{$\mathcal{L}_{h_{k}} = \mathcal{L}_{h_{k+1}}$}
\label{fig16}
\end{figure}

\begin{proposition}
The elements in the sequence $\{h_{i} \}_{i=1}^{\infty}$ 
are all from distinct conjugacy classes of $F.$ %$[h_j] \not\sim_{conj} [h_k]$ for $k \neq j.$
\label{prop34}
\end{proposition}
\begin{proof}{}
%
%We prove by an inductive argument where 
%
%
We first show that $h_1$ and $h_2$ 
are in different conjugacy classes. By Theorem \ref{conj}, 
it is equivalent to show that the reduced annular strand diagrams 
 $r\mathcal{A}_{h_1}$ and $r\mathcal{A}_{h_2}$ produced 
 from $h_1$ and $h_2$ respectively, 
 are different.

From Fig. \ref{fig17}, 
we see that $\mathcal{A}_{h_{2}}$ can be reduced to $\mathcal{A}_{h_1} \sqcup \mathcal{E}$, 
a disjoint union of $\mathcal{A}_{h_1}$ and 
a non-trivial reduced annular strand diagram $\mathcal{E}.$ 
%where we can write as $\mathcal{A}_{h_0} \sqcup \mathcal{E}.$ 
%
We perform reductions to $\mathcal{A}_{h_1}$ 
to obtain a reduced annular strand diagram $r\mathcal{A}_{h_1}.$ 
Since $r\mathcal{A}_{h_1} \sqcup \mathcal{E}$ is reduced, 
we have
$r\mathcal{A}_{h_2} = r\mathcal{A}_{h_1} \sqcup \mathcal{E}$. 
It is obvious that $r\mathcal{A}_{h_2}$ has one more component 
than $r\mathcal{A}_{h_1}$ as graphs, 
and thus $h_1$ and $h_2$ belong to different conjugacy classes.
Inductively, we see that the number of components of $r\mathcal{A}_{h_i}$ increases as $i$ increases. 
Therefore, $h_i$'s are all in different conjugacy classes of $F$. 
\end{proof}

\begin{figure}[htp]
\centerline{\includegraphics[width=4.5in]{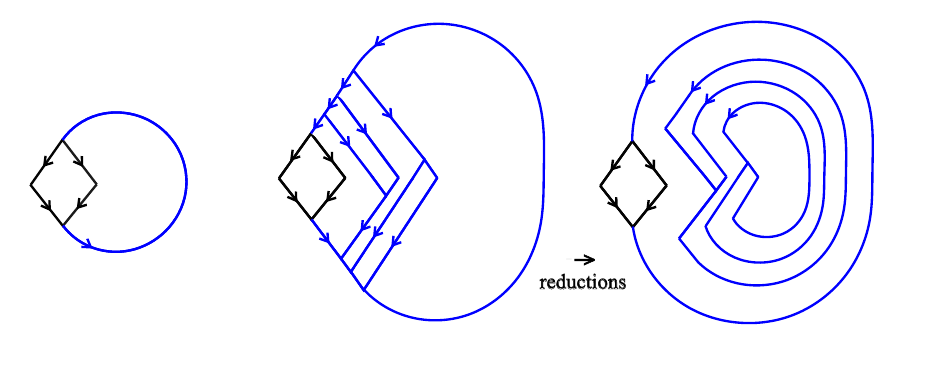}}
\caption{$r\mathcal{A}_{h_2} = r\mathcal{A}_{h_1} \sqcup \mathcal{E}$}
\label{fig17}
\end{figure}

Now we are ready to prove the main theorem of this section.
\begin{theorem}
For any unoriented link $L $ there exist elements 
$\{h_{i}\}_{i = 1}^{\infty}$ in $F$ such that 
\begin{enumerate}
\item $\mathcal{L}_{h_i} = L$ for all $i \in \mathbb{N},$
\item $h_i$ are all in distinct conjugacy classes. 
\end{enumerate}
\label{thm1}
\end{theorem}

\begin{proof}{}
By Theorem \ref{jones}, there is a reduced tree diagram $h_1=(T_+, T_-)$ 
for which $\mathcal{L}_{h_1} = L$. 
We then construct a sequence $\{h_{i}\}_{i = 1}^{\infty}$ as in Fig. \ref{fig15}. 
The remaining part of the proof follows directly 
from Proposition \ref{prop33} and Proposition \ref{prop34}. 
\end{proof}
 %%%%%%%%%%%%%%%%%%%%%%%%%%%%%%%
 
%
\section{Distinct links constructed from the same conjugacy class of $F$}
\label{sec4}

As a problem complementary to Theorem \ref{thm1}, 
we want to know whether we could construct all knots and links from elements 
in one conjugacy class. 
%An easier task is to consider 
A weaker version of the task is to find out 
examples of sequences of distinct links constructed from elements 
in the same conjugacy class.
Here we study the conjugacy classes of the generators $x_0$ and $x_1$. 
Note that other generators $x_i$ for $i\geq 2$ are all conjugate to $x_1$. 

\begin{theorem}
\label{thm2}
For $i=0, 1$, the elements in the conjugacy class of $x_i$ produce 
infinitely many distinct knots and links. 
More precisely, they produce $2$-bridge links $C(1, 1, \cdots, 1)$ 
with even number of $1$s.
\end{theorem}

\subsection{The proof for the case of $x_0$}
The proof is given by concrete constructions.   
For a reduced tree diagram, when the target tree has the form in Fig. \ref{f1}, 
we say it is a {\it positive element} in $F$ \cite{Aiello:2021aa}. 
The set of all positive elements is denoted by $F_{+}$ which form a monoid.  

\begin{figure}[htp]
\begin{center}
\begin{tikzpicture}[baseline=-0.65ex, thick, scale=0.4]
%\draw (0, 4) to (-4, 0);
%\draw (0, 4) to (4, 0);
%\draw (3, 1) to (2, 0);
%\draw (-3, 1) to (-2, 0);
%\draw (-2, 2) to (0, 0);
\draw (0, -5) to (-5, 0);
\draw (0, -5) to (5, 0);
\draw (1, -4) to (-3, 0);
\draw (2, -3) to (-1, 0);
\draw (4, -1) to (3, 0);
\draw   (2, -1) node{$\cdots$};
\end{tikzpicture}
\end{center}
 \caption{The target tree of a positive element.}
 \label{f1}
 \end{figure}
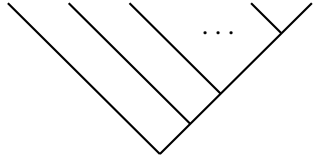

Then for any $g = (T_{+}, T_{-}) \in F_{+}$, 
the source (resp. target) tree of the tree diagram representing $gx_0g^{-1}$ is obtained from $T_{+}$ 
by attaching a caret on its leftmost (resp. rightmost) leaf, 
which is illustrated in Fig. \ref{f2}.

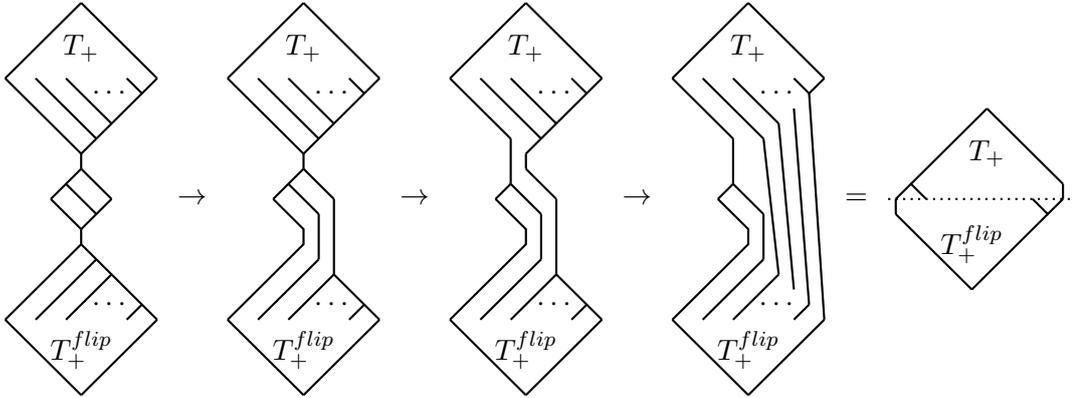
\begin{figure}[htp]
\begin{center}
\begin{tikzpicture}[baseline=-0.65ex, thick, scale=0.2]
\draw (0, 5) to (-5, 0);
\draw (0, 5) to (5, 0);
\draw (0, -5) to (-5, 0);
\draw (0, -5) to (5, 0);
\draw (1, -4) to (-3, 0);
\draw (2, -3) to (-1, 0);
\draw (4, -1) to (3, 0);
\draw   (2, -1) node{$\cdots$};
\draw   (0, 2) node{$T_+$};
\draw (0, -6) to (-2,-8);
\draw (0, -6) to (2, -8);
\draw (0, -10) to (-2,-8);
\draw (0, -10) to (2, -8);
\draw (-1, -7) to (1, -9);
\draw (0, -6) to (0, -5);
\draw (0, -11) to (-5, -16);
\draw (0, -11) to (5, -16);
\draw (1, -12) to (-3, -16);
\draw (2, -13) to (-1, -16);
\draw (4, -15) to (3, -16);
\draw   (2, -15) node{$\cdots$};
\draw   (0, -18) node{$T_+^{flip}$};
\draw (0, -21) to (-5, -16);
\draw (0, -21) to (5, -16);
\draw (0, -11) to (0, -10);
\end{tikzpicture}
\begin{tikzpicture}[baseline=-0.65ex, thick, scale=0.2]
\draw (0, -8) node{$\to$};
\end{tikzpicture}
\begin{tikzpicture}[baseline=-0.65ex, thick, scale=0.2]
\draw (0, 5) to (-5, 0);
\draw (0, 5) to (5, 0);
\draw (0, -5) to (-5, 0);
\draw (0, -5) to (5, 0);
\draw (1, -4) to (-3, 0);
\draw (2, -3) to (-1, 0);
\draw (4, -1) to (3, 0);
\draw   (2, -1) node{$\cdots$};
\draw   (0, 2) node{$T_+$};
\draw (0, -6) to (-2,-8);
\draw (0, -6) to (2, -8);
\draw (-2, -8) to (0,-10) to (0, -11) to (-5,-16);
%\draw (1, -9) to (2, -8);
\draw (-1, -7) to (1, -9);
\draw (0, -6) to (0, -5);
\draw (1, -12) to (1, -9);
\draw (2, -13) to (5, -16);
\draw (1, -12) to (-3, -16);
\draw (2, -13) to (-1, -16);
\draw (4, -15) to (3, -16);
\draw   (2, -15) node{$\cdots$};
\draw   (0, -18) node{$T_+^{flip}$};
\draw (0, -21) to (-5, -16);
\draw (0, -21) to (5, -16);
\draw (2, -8) to (2, -13);
\end{tikzpicture}
\begin{tikzpicture}[baseline=-0.65ex, thick, scale=0.2]
\draw (0, -8) node{$\to$};
\end{tikzpicture}
\begin{tikzpicture}[baseline=-0.65ex, thick, scale=0.2]
\draw (0, 5) to (-5, 0);
\draw (0, 5) to (5, 0);
\draw (-1, -4) to (-5, 0);
\draw (0, -5) to (5, 0);
\draw (1, -4) to (-3, 0);
\draw (2, -3) to (-1, 0);
\draw (4, -1) to (3, 0);
\draw (-1, -4) to (-1, -7);
\draw   (2, -1) node{$\cdots$};
\draw   (0, 2) node{$T_+$};
\draw (-1, -7) to (-2,-8);
\draw (0, -6) to (2, -8);
\draw (-2, -8) to (0,-10) to (0, -11) to (-5,-16);
%\draw (1, -9) to (2, -8);
\draw (-1, -7) to (1, -9);
\draw (0, -6) to (0, -5);
\draw (1, -12) to (1, -9);
\draw (2, -13) to (5, -16);
\draw (1, -12) to (-3, -16);
\draw (2, -13) to (-1, -16);
\draw (4, -15) to (3, -16);
\draw   (2, -15) node{$\cdots$};
\draw   (0, -18) node{$T_+^{flip}$};
\draw (0, -21) to (-5, -16);
\draw (0, -21) to (5, -16);
\draw (2, -8) to (2, -13);
\end{tikzpicture}
\begin{tikzpicture}[baseline=-0.65ex, thick, scale=0.2]
\draw (0, -8) node{$\to$};
\end{tikzpicture}
\begin{tikzpicture}[baseline=-0.65ex, thick, scale=0.2]
\draw (0, 5) to (-5, 0);
\draw (0, 5) to (5, 0);
\draw (-1, -4) to (-5, 0);
\draw (4, -1) to (5, 0);
\draw (1, -4) to (-3, 0);
\draw (2, -3) to (-1, 0);
\draw (4, -1) to (3, 0);
\draw (-1, -4) to (-1, -7);
\draw (4, -1) to (5, -16);
\draw   (2, -1) node{$\cdots$};
\draw   (0, 2) node{$T_+$};
\draw (-1, -7) to (-2,-8);
%\draw (0, -6) to (2, -8);
\draw (-2, -8) to (0,-10) to (0, -11) to (-5,-16);
%\draw (1, -9) to (2, -8);
\draw (-1, -7) to (1, -9);
%\draw (0, -6) to (0, -5);
\draw (1, -12) to (1, -9);
%\draw (2, -13) to (5, -16);
\draw (1, -12) to (-3, -16);
\draw (2, -13) to (-1, -16);
\draw (4, -15) to (3, -16);
\draw   (2, -15) node{$\cdots$};
\draw   (0, -18) node{$T_+^{flip}$};
\draw (0, -21) to (-5, -16);
\draw (0, -21) to (5, -16);
%\draw (3, -7) to (2, -13);
\draw (1, -4) to (2, -13);
\draw (2, -3) to (3, -14);
\draw (3, -2) to (4, -15);
\end{tikzpicture}
\begin{tikzpicture}[baseline=-0.65ex, thick, scale=0.2]
\draw (0, -8) node{$=$};
\end{tikzpicture}
\begin{tikzpicture}[baseline=-0.65ex, thick, scale=0.2]
\draw (1, -2) to (-4, -7);
\draw (1, -2) to (6, -7);
\draw   (1, -5) node{$T_+$};
\draw   (0, -11) node{$T_+^{flip}$};
\draw (0, -14) to (-5, -9);
\draw (0, -14) to (5, -9);
\draw  (-5, -8) to (-4, -7) to (-3, -8);
\draw (4, -8) to (5, -9) to (6, -8);
\draw (-5, -8) to (-5, -9);
\draw (6, -8) to (6, -7);
\draw [dotted] (-5.5, -8) to (6.5, -8);
\end{tikzpicture}
 \caption{The tree diagram representing $gx_0g^{-1}$. \label{f2}}
 \end{center}
 \end{figure}

Now we consider a sequence of tree diagrams $\{g_n\}_{n=0}^{\infty}\subset F_{+}$ 
whose source tree $T_n$ is defined to be the tree with $2n+2$ leaves 
as shown in Fig. \ref{f3}. 
More precisely, let $T_0$ be the unique binary tree with one root and two leaves, 
and then for $n\in \mathbb{N}$, $T_{n}$ is obtained from $T_{n-1}$ 
by attaching the root of the tree with thickend edges in Fig. \ref{f3} to the rightmost leaf of $T_{n-1}$.

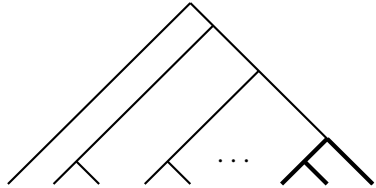
\begin{figure}[htp]
\begin{center}
\begin{tikzpicture}[baseline=-0.65ex, thick, scale=0.3]
\draw (-1, 8) to (7, 0);
\draw (-1, 8) to (-9, 0);
\draw (0, 7) to (-7, 0);
\draw (-5, 0) to (-6, 1);
\draw (2, 5) to (-3, 0);
%\draw (1, 0) to (0, 1);
\draw (-1, 0) to (-2, 1);
\draw [ultra thick] (5, 2) to (3, 0);
\draw [ultra thick] (4, 1) to (5, 0);
\draw [ultra thick] (5, 2) to (7,0);
\draw (1, 1) node{$\cdots$};
\end{tikzpicture}
\caption{The tree $T_n$. \label{f3}}
\end{center}
\end{figure}

%Here we require that $n$ is even and $n \geq 4$. %?
Then the tree diagram representing $g_nx_0g_n^{-1}$ 
and the associated Tait graph are given in Fig. \ref{f4}. 
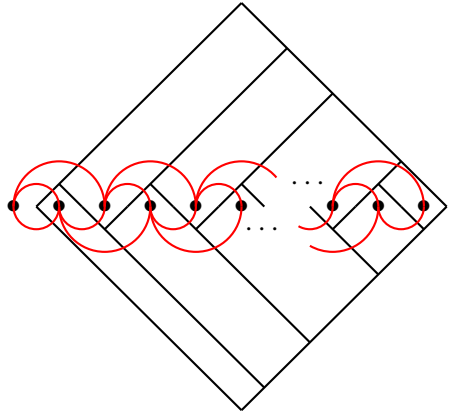
\begin{figure}[htp]
\begin{center}
\begin{tikzpicture}[baseline=-0.65ex, thick, scale=0.3]
\draw (-2, 9) to (7, 0);
\draw (-2, 9) to (-11, 0);
\draw (0, 7) to (-7, 0);
\draw (-5, 0) to (-6, 1);
\draw (2, 5) to (-3, 0);
%\draw (1, 0) to (0, 1);
\draw (-1, 0) to (-2, 1);
\draw (5, 2) to (3, 0);
\draw (4, 1) to (5, 0);
\draw (-9, 0) to (-10, 1);
\draw (1, 1) node{$\cdots$};
\draw (-2, -9) to (7, 0);
\draw (-2, -9) to (-11, 0);
\draw (-1, -8) to (-9, 0);
\draw (-7, 0) to (-8, -1);
\draw (1, -6) to (-5, 0);
%\draw (1, 0) to (0, 1);
\draw (-3, 0) to (-4, -1);
\draw (4, -3) to (1, 0);
\draw (2, -1) to (3, 0);
\draw (6, -1) to (5, 0);
\draw (-1, -1) node{$\cdots$};
\draw (-12, 0) node{$\bullet$};
\draw (-10, 0) node{$\bullet$};
\draw (-8, 0) node{$\bullet$};
\draw (-6, 0) node{$\bullet$};
\draw (-4, 0) node{$\bullet$};
\draw (-2, 0) node{$\bullet$};
\draw (2, 0) node{$\bullet$};
\draw (4, 0) node{$\bullet$};
\draw (6, 0) node{$\bullet$};
\draw [red] (-10,0) arc[start angle=0, end angle=180, radius=1cm];
\draw [red](-8,0) arc[start angle=0, end angle=180, radius=2cm];
\draw [red](-6,0) arc[start angle=0, end angle=180, radius=1cm];
\draw [red](-4,0) arc[start angle=0, end angle=180, radius=2cm];
\draw [red](-2,0) arc[start angle=0, end angle=180, radius=1cm];
\draw [red](-4,0) arc[start angle=-180, end angle=-320, radius=2cm];
\draw [red](4,0) arc[start angle=0, end angle=180, radius=1cm];
\draw  [red] (6,0) arc[start angle=0, end angle=180, radius=2cm];
\draw [red] (-10,0) arc[start angle=0, end angle=-180, radius=1cm];
\draw [red](-8,0) arc[start angle=0, end angle=-180, radius=1cm];
\draw [red](-6,0) arc[start angle=0, end angle=-180, radius=2cm];
\draw [red](-4,0) arc[start angle=0, end angle=-180, radius=1cm];
\draw [red](-2,0) arc[start angle=0, end angle=-180, radius=2cm];
%\draw [red](-8,0) arc[start angle=0, end angle=-180, radius=1cm];
\draw [red](6,0) arc[start angle=0, end angle=-180, radius=1cm];
\draw [red](2,0) arc[start angle=0, end angle=-120, radius=1cm];
\draw [red](4,0) arc[start angle=0, end angle=-120, radius=2cm];
\end{tikzpicture}
\end{center}
\caption{The tree diagram representing $g_nx_0g_n^{-1}$ and the associated Tait graph. \label{f4}}
\end{figure}

The sequence of links corresponding to the Tait graphs in Fig. \ref{f4} is 
the $2$-bridge link $C(1, 1, \cdots, 1)$ in Conway normal form, 
where the number of $1$s is $2n$. 
See Fig. \ref{f51} for the case of $n=3$. 
In this way, we obtain infinitely many different $2$-bridge links constructed 
from elements belonging to the conjugacy class of $x_0$.

\begin{figure}[htp]
\centerline{
\includegraphics[width=4.5in]{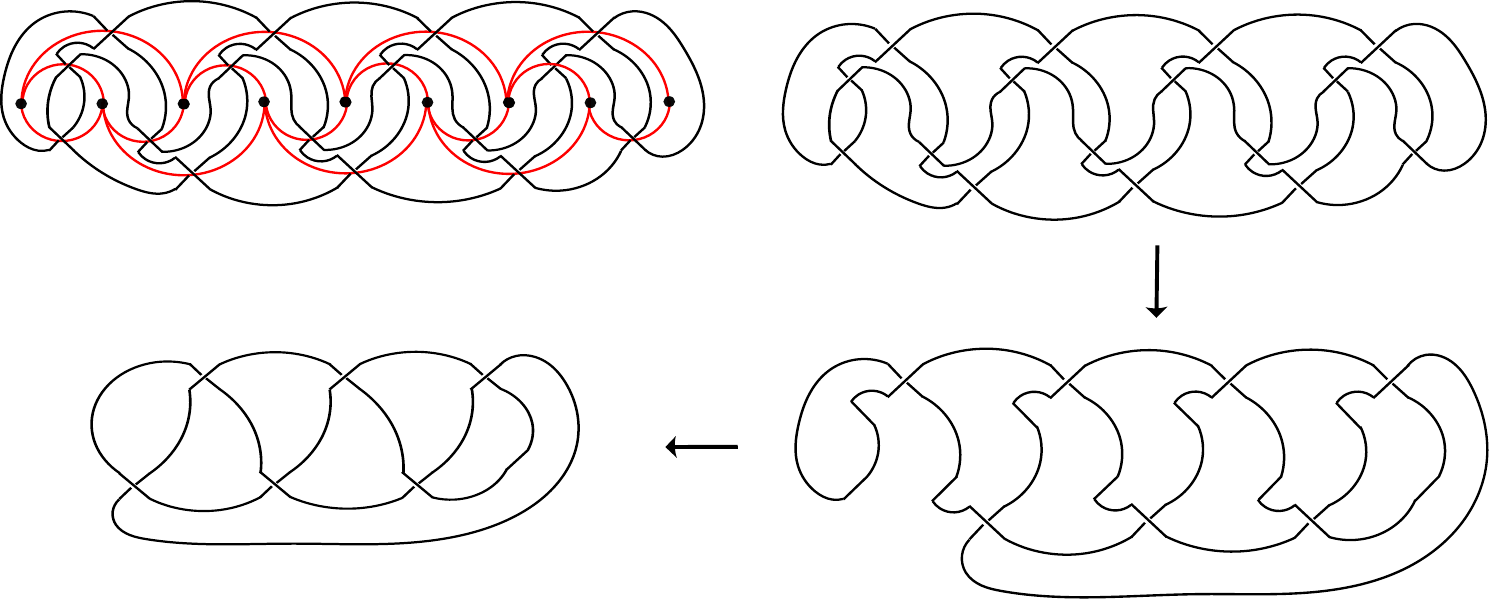}}
\caption{The link constructed from $g_3x_0g_3^{-1}$ is the 2-bridge link $C(1, 1, 1,1,1, 1)$. \label{f51}}
\end{figure}

\subsection{The proof for the case of $x_1$.}
For the generator $x_1$ and a positive element $g=(T_{+}, T_{-})\in F_{+}$, 
we see that the source tree of the tree diagram representing $gx_1g^{-1}$ 
is obtained from $T_{+}$ 
by attaching a caret to its $2$nd leaf from the left, 
while its target tree is obtained from $T_{+}$ by attaching a caret on its rightmost leaf. 
The discussion can be conducted similarly as in the case of $x_0$. 

Consider a positive element $h_n$ whose source tree $S_n$ 
is defined as in Fig. \ref{f5} (left), 
which is obtained from that of $g_n$ 
by attaching the root of $T_n$ to the right leaf of a caret. 
Then the tree diagram $h_n x_1 h_n^{-1}$ is obtained 
from $g_nx_0g_n^{-1}$ by adding a pair of carets, 
as shown in Fig. \ref{f5} (right).

\begin{figure}[htp]
\begin{center}
\begin{tikzpicture}[baseline=-0.65ex, thick, scale=0.3]
\draw (-1, 8) to (7, 0);
\draw (-1, 8) to (-9, 0);
\draw (0, 7) to (-7, 0);
\draw (-5, 0) to (-6, 1);
\draw (2, 5) to (-3, 0);
%\draw (1, 0) to (0, 1);
\draw (-1, 0) to (-2, 1);
\draw (5, 2) to (3, 0);
\draw (4, 1) to (5, 0);
\draw (1, 1) node{$\cdots$};
\draw [ultra thick](-1, 8) to (-2, 9);
\draw [ultra thick](-11, 0) to (-2, 9);
\draw (-3, -2) node{$S_n$};
\end{tikzpicture} \quad \quad
\begin{tikzpicture}[baseline=-0.65ex, thick, scale=0.25]
\draw (-1, 8) to (5, 2);
\draw (-1, 8) to (-7, 2);
\draw [ultra thick](-1, 8) to (-2, 9);
\draw [ultra thick](-9, 2) to (-2, 9);
\draw (-1, -4) to (5, 2);
\draw (-1, -4) to (-7, 2);
\draw [ultra thick](-1, -4) to (-2, -5);
\draw [ultra thick](-9, 2) to (-2, -5);
\draw (-1, 2) node{$g_nx_0g_n^{-1}$};
\end{tikzpicture}
\caption{$S_n$ (left) is obtained from $T_n$ by adding two edges (thick edges). \label{f5}}
\end{center}
\end{figure}
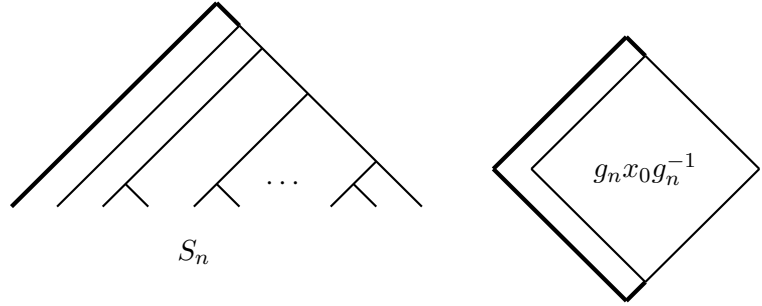

From the construction of Tait graph in Section \ref{sec2.2}, 
we see that the Tait graph associated with the tree diagram of $h_n x_1 h_n^{-1}$ is obtained from 
that of $g_n x_0 g_n^{-1}$ 
by adding a vertex and two edges as in Fig. \ref{f6}. 
The corresponding link is the union of an unknot 
and the $2$-bridge link $C(1, 1, \cdots, 1)$ constructed from $g_n x_0 g_n^{-1}$. 
In Jones's construction, two links are usually identified 
if they only differ by finitely many trivial components. 
As a result, we see that the $2$-bridge link $C(1, 1, \cdots, 1)$ 
can be constructed from elements in the conjugacy class of $x_1$ as well.

\begin{figure}[htp]
\begin{center}
\begin{tikzpicture}[baseline=-0.65ex, thick, scale=0.6]
\draw (-1, 0) node{$\bullet$};
\draw (1, 0) node{$\bullet$};
\draw (3, 0) node{$\bullet$};
\draw (5, 0) node{$\bullet$};
\draw [red] (7,0) node{$\cdots$};
\draw [red, ultra thick] (1,0) arc[start angle=0, end angle=180, radius=1cm];
\draw [red, ultra thick] (1,0) arc[start angle=0, end angle=-180, radius=1cm];
\draw [red] (3,0) arc[start angle=0, end angle=180, radius=1cm];
\draw [red] (3,0) arc[start angle=0, end angle=-180, radius=1cm];
\draw [red] (5,0) arc[start angle=0, end angle=180, radius=2cm];
\draw [red] (5,0) arc[start angle=0, end angle=-180, radius=1cm];
\draw [red] (3,0) arc[start angle=-180, end angle=-30, radius=2cm];
\draw (-0.5, 0.5) to (0.5, 1.5);
\draw (-0.5, -0.5) to (0.5, -1.5);
\draw (0.5, 0.5) to (0.2, 0.8);
\draw (0.5, -0.5) to (0.2, -0.8);
\draw (-0.2, 1.2) to (-0.5, 1.5);
\draw (-0.2, -1.2) to (-0.5, -1.5);
\draw (-0.5, 0.5) to (-0.5, -0.5);
\draw (0.5, 0.5) to (0.5, -0.5);
\draw (-0.5, 1.5) arc[start angle=90, end angle=270, radius=1.5cm];
\end{tikzpicture}
\end{center}
\caption{The Tait graph of $h_n x_1 h_n^{-1}$ is obtained from that of $g_n x_0 g_n^{-1}$ 
by adding one vertex and two edges (thick red). \label{f6}}
\end{figure}
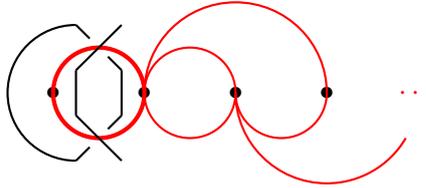

%%%%%%%%%%%%%%%%%%%%%%%%%%%%%%%%%%%%%%%%%%%%
%
\bigskip
\section*{Acknowledgments}
We would like to thank Masaaki Suzuki, Makoto Sakuma, Sebastian Baader, Valeriano Aiello,
 John Parker, Andrew Lobb and Sadayosi Kojima 
for all the valuable comments and suggestions throughout the project. 
The second named author is grateful for all the financial support provided 
by Mathematical and Theoretical Physics Unit, OIST. 
The first named author is partially supported by Grant-in-Aid JP20K14304.
\bigskip
\bibliographystyle{plain}
\bibliography{ref}{}

\end{document}